\documentclass{article}
\usepackage[left=4cm, right=4cm, top=4cm]{geometry} 
\usepackage{amsmath,amsthm,amssymb,mathrsfs,amstext, titlesec}
\makeatletter

\titleformat{\section}
{\normalfont\Large\bfseries}{\thesection}{1em}{}
\titleformat*{\subsection}{\Large\bfseries}
\titleformat*{\subsubsection}{\large\bfseries}
\titleformat*{\paragraph}{\large\bfseries}
\titleformat*{\subparagraph}{\large\bfseries}

\renewcommand{\@seccntformat}[1]{\csname the#1\endcsname. }
\makeatother
\theoremstyle{plain}
\newtheorem{theorem}{Theorem}

\theoremstyle{definition}
\newtheorem{definition}[theorem]{Definition}

\providecommand{\keywords}[1]{{\bf{Keywords:}} #1}

\title{An infinitary Zero sum theorem
\footnote{\keywords{Zero sum theory, Ramsey Theory, Algebra of the Stone-\v{C}ech compactification}}}

\date{}
\author{Sayan Goswami
\footnote{Department of Mathematics, 
           University of Kalyani, 
           Kalyani-741235,
           Nadia, West Bengal, India
           {\bf sayan92m@gmail.com}
           }
           }

\begin{document}
\maketitle

\begin{abstract}
\noindent Erd\H{o}s-Ginzburg-Ziv theorem says that if there are $2n-1$ number is
given, then there are $n$ numbers such that their sum is divided
by $n$. We will connect this theorem with the Ramsey theoretic large
sets and will prove an infinitary version of this theorem. In our proof we will use the methods of ultrafilters. But one may proceed using methods of Topological dynamics.
\end{abstract} 


Erd\H{o}s-Ginzburg-Ziv theorem \cite{key-1} says that if we have $2n-1$
numbers, then there exists $n$ numbers among them whoose sum congruent
to $0$ modulo $n$. If $r,n\in\mathbb{N}$, then for any $r$-partition
of a finite set of $\left(2n-1\right)r$ numbers there is a set having
$\left(2n-1\right)$ elements and so there is a cell containing $n$
elements whose sum is divisible by $n$. Instead of taking $\left(2n-1\right)r$ elements, one can proceed using $rn^2$ elements. Now we will provide an infinitary
version of this theorem. Here we will provide a joint extension of the
Zero sum theorem and Hindman's theorem \cite{key-3}. We
are motivated by the Central Sets Theorem \cite{key-2}, established
by H. Furstenberg, which is known as joint extension of van der Waerden's
theorem \cite{key-6} and Hindman's theorem. For details on Central
Sets Theorem , one can see \cite{key-5}. Throughout the article,
$\mathcal{P}_{f}\left(\mathbb{N}\right)$ denotes the set of all non-empty
finite subsets of $\mathbb{N}$.

As it will be needed in our proof, we now give a brief review about
the Stone-\v{C}ech compactification of a discrete semigroup. For
details readers are invited to read \cite{key-4}. Let $\left(S,\cdot\right)$
be any discrete semigroup and denote its Stone-\v{C}ech compactification
by $\beta S$. $\beta S$ is the set of all ultrafilters on $S$,
where the points of $S$ are identified with the principal ultrafilters.
The basis for the topology is $\left\{ \bar{A}:A\subseteq S\right\} $,
where $\bar{A}=\left\{ p\in\beta S:A\in p\right\} $. The operation
of $S$ can be extended to $\beta S$ making $\left(\beta S,\cdot\right)$
a compact, right topological semigroup with $S$ contained in its
topological center. That is, for all $p\in\beta S$, the function
$\rho_{p}:\beta S\rightarrow\beta S$ is continuous, where $\rho_{p}\left(q\right)=q\cdot p$
and for all $x\in S$, the function $\lambda_{x}:\beta S\rightarrow\beta S$
is continuous, where $\lambda_{x}\left(q\right)=x\cdot q$. For $p,q\in\beta S$
and $A\subseteq S$, $A\in p\cdot q$ if and only if $\left\{ x\in S:x^{-1}A\in q\right\} \in p$,
where $x^{-1}A=\left\{ y\in S:x\cdot y\in A\right\} $. 

Since $\beta S$ is a compact Hausdorff right topological semigroup,
it has a smallest two sided ideal denoted by $K\left(\beta S\right)$,
which is the union of all of the minimal right ideals of $S$, as
well as the union of all of the minimal left ideals of $S$. Every
left ideal of $\beta S$ contains a minimal left ideal and every right
ideal of $\beta S$ contains a minimal right ideal. The intersection
of any minimal left ideal and any minimal right ideal is a group,
and any two such groups are isomorphic. Any idempotent $p$ in $\beta S$
is said to be minimal if and only if $p\in K\left(\beta S\right)$.
\begin{definition}
Let $\left(S,\cdot\right)$ be a semigroup and $A\subseteq S$, then
\end{definition}

\begin{enumerate}
\item The set $A$ is thick if and only if for any finite subset $F$ of
$S$, there exists an element $x\in S$ such that $F\cdot x\subset A$.
This means the sets which contains a translation of any finite subset.
For example one can see $\cup_{n\in\mathbb{N}}\left[2^{n},2^{n}+n\right]$
is a thick set in $\mathbb{N}$.
\item The set $A$ is syndetic if and only if there exists a finite subset
$G$ of $S$ such that $\bigcup_{t\in G}t^{-1}A=S$. That is, with
a finite translation if, the set which covers the entire semigroup,
then it will be called a Syndetic set. For example the set of even
and odd numbers are both syndetic in $\mathbb{N}$.
\item The sets which can be written as an intersection of a syndetic and
a thick set are called $\mathit{Piecewise}$ $\mathit{syndetic}$
sets. More formally a set $A$ is $\mathit{Piecewise}$ $\mathit{syndetic}$
if and only if there exists $G\in\mathcal{P}_{f}\left(S\right)$ such
that for every $F\in\mathcal{P}_{f}\left(S\right)$, there exists
$x\in S$ such that $F\cdot x\subseteq\bigcup_{t\in G}t^{-1}A$. Clearly
the thick sets and syndetic sets are natural examples of $\mathit{Piecewise}$
$\mathit{syndetic}$ sets. From definition one can immediately see
that $2\mathbb{N}\cap\bigcup_{n\in\mathbb{N}}\left[2^{n},2^{n}+n\right]$
is a nontrivial example of $\mathit{Piecewise}$ $\mathit{syndetic}$
sets in $\mathbb{N}$.
\item Then a subset $A$ of $S$ is called central if and only if there
is some minimal idempotent $p$ such that $A\in p$.
\end{enumerate}
\begin{theorem}
Let $m,n\in\mathbb{N}$, and $\langle x_{i,1}\rangle_{i=1}^{\infty}$,
$\langle x_{i,2}\rangle_{i=1}^{\infty}$,$\ldots$, $\langle x_{i,m}\rangle_{i=1}^{\infty}$
be $m$ distinct sequences. If $B\subseteq\mathbb{N}$ be a central
set, then for each $j\in\left\{ 1,2,\ldots,m\right\} $, there exists
sequences of finite sets $\left(F_{i,j}\right)_{i=1}^{\infty}\subseteq\langle x_{i,j}\rangle_{i=1}^{\infty}$
such that $F_{i_{1},j}\cap F_{i_{2},j}=\emptyset$ for $i_{1}\neq i_{2}$,
and a sequence of elements $\left(z_{i,j}\right)_{i=1}^{\infty}$
in $\mathbb{N}$ such that, for each $j\in\left\{ 1,2,\ldots,m\right\} $,
\begin{enumerate}
\item $\mid F_{i,j}\mid=n$, for each $i\in\mathbb{N}$,
\item $\sum_{t\in F_{i,j}}t\equiv0\,\left(\mod\,n\right)$ for each $i\in\mathbb{N}$,
\item $\left(F_{i_{1},j_{1}}+z_{i_{1},j_{1}}\right)+\left(F_{i_{2},j_{2}}+z_{i_{2},j_{2}}\right)+\cdots+\left(F_{i_{k},j_{k}}+z_{i_{k},j_{k}}\right)\subseteq B$
for each $i_{1}<i_{2}<\ldots<i_{k}$ and $j_{p}\in\left\{ 1,2,\ldots,m\right\} $
for each $p$.
\end{enumerate}
\end{theorem}

\begin{proof}
Let $p\in\beta\mathbb{N}$ be a minimal idempotent such that $B\in p$.
As $B\in p,$
we have $B^{*}=\left\{ x\in B:-x+B\in p\right\} \in p$ and so for $x\in B^{*}$,
$-x+B^{*}\in p$. As $B^{*}\in p,$it is piecewise syndetic. So there
exists a finite set $G$ such that $\cup_{t\in G}\left(-t+B^{*}\right)$
is a thick set. Let $\mid G\mid=r$, hence for the finite set $E=\left\{ \langle x_{i,1}\rangle_{i=1}^{\left(2n-1\right)r},\langle x_{i,2}\rangle_{i=1}^{\left(2n-1\right)r},,\ldots,\langle x_{i,m}\rangle_{i=1}^{\left(2n-1\right)r}\right\} $,
there exists an element $x\in\mathbb{N}$ such that $E+x\subset\cup_{t\in G}\left(-t+B^{*}\right).$
Hence for each $j\in\left\{ 1,2,\ldots,m\right\} $, there exists
$\left(2n-1\right)$ elements from $\langle x_{i,j}\rangle_{i=1}^{\left(2n-1\right)r}$
such that its translation by $z_{1,j}=x+s$, for some $s\in G$, is
contained in $B^{*}$. Hence from Erd\H{o}s-Ginzburg-Ziv theorem, there
exists a finite set $F_{1,j}$ of $n$ elements such that $\sum_{t\in F_{1,j}}t\equiv0\,\left(\mod\,n\right)$
and by construction $F_{1,j}+z_{1,j}\in B^{*}$. Let us assume that
for $i=1,2,\ldots,l$, the above configuration lies in $B^{*}$ by
the induction hypothesis, i.e. 

1. $\mid F_{i,j}\mid=n$, for each $i\in\left\{ 1,2,\ldots,l\right\} $,

2. $\sum_{t\in F_{i,j}}t\equiv0\,\left(\mod\,n\right)$ for each $i\in\left\{ 1,2,\ldots,l\right\} $,


3. $\left(F_{i_{1},j_{1}}+z_{i_{1},j_{1}}\right)+\left(F_{i_{2},j_{2}}+z_{i_{2},j_{2}}\right)+\cdots+\left(F_{i_{k},j_{k}}+z_{i_{k},j_{k}}\right)\subseteq B^{*}$
for each $i_{1}<i_{2}<\ldots<i_{k}$ in $\left\{ 1,2,\ldots,l\right\} $
and $j_{p}\in\left\{ 1,2,\ldots,m\right\} $ for each $p\in\left\{ 1,2,\ldots,l\right\} $.

\noindent Let $Z=\left\{ \sum_{p=1}^{k}\left(F_{i_{p},j_{p}}+z_{i_{p},j_{p}}\right):i_{1}<\ldots<i_{k}\in\left\{ 1,\ldots,l\right\} ,j_{p}\in\left\{ 1,\ldots,m\right\} \right\} $.
Now choose $C=B^{*}\cap\bigcap_{z\in Z}\left(-z+B^{*}\right)\in p$.
Then choosing sufficiently large blocks of numbers choosing from the
$m$ sequences and using the same argument as above there exists finite
sets $F_{\left(l+1\right),j}$ of $n$ elements such that $\sum_{t\in F_{\left(l+1\right),j}}t\equiv0\,\left(\mod\,n\right)$
and by induction hypothesis $\left(F_{i_{1},j_{1}}+z_{i_{1},j_{1}}\right)+\left(F_{i_{2},j_{2}}+z_{i_{2},j_{2}}\right)+\cdots+\left(F_{i_{k},j_{k}}+z_{i_{k},j_{k}}\right)\in B^{*}$
for each $i_{1}<i_{2}<\ldots<i_{k}$ in $\left\{ 1,2,\ldots,l+1\right\} $
and $j_{p}\in\left\{ 1,2,\ldots,m\right\} $ for each $p\in\left\{ 1,2,\ldots,l+1\right\} $. This completes the induction.
\end{proof}

\end{document}